\documentclass{article}
\usepackage{tikz}
\usepackage{grfext}
\usepackage{listings}
\usepackage{color}
\usepackage{graphicx}           
\usepackage{amsthm,amsfonts,amsmath,amssymb}
\usepackage{array,blkarray,multirow, enumerate}
\usepackage{xcolor}
\usepackage{authblk}
\usepackage{caption, setspace}

\DeclareMathSymbol{\mrq}{\mathord}{operators}{`'}

\newtheorem{theorem}{Theorem}
\newtheorem{lemma}{Lemma}

\newtheorem*{obs*}{Observation}

\newtheorem{definition}{Definition}
\newtheorem{claim}{Claim}

\newcounter{mycount}

\def\0{\mathbf{0}}

\usepackage[normalem]{ulem}

\usepackage{lipsum}

\begin{document}
\lstset{language=Python}          

\title{Structure and substructure connectivity of folded divide-and-swap cube}

\author[1]{Muhammed Türkmen \footnote{E-mail:muhammedturkmen831@gmail.com}}
\author[2]{Canan \c{C}ift\c{c}i \footnote{E-mail:cananciftci@odu.edu.tr}}
\author[3] {G\"{u}lnaz Boruzanl{\i} Ekinci \footnote{E-mail:gulnaz.boruzanli@ege.edu.tr}\footnote{Corresponding author} }
\affil[1,2]{Department of Mathematics, Faculty of Arts and Sciences\qquad\qquad\qquad\qquad\qquad Ordu University, Ordu, Türkiye}
\affil[3]{Department of Mathematics, Faculty of Science\qquad\qquad\qquad\qquad\qquad\qquad\qquad\qquad\qquad\qquad Ege University, Izmir, Türkiye}
\date{}

\maketitle
\begin{abstract}
Let $ \mathcal{H} $ be a connected subgraph of a graph $ G $. The structure connectivity of $ G $, denoted by $ \kappa(G;\mathcal{H}) $, is the minimum number of a set of connected subgraphs in $ G $, whose removal disconnects  $ G $ and each element in the set is isomorphic to $ \mathcal{H} $. The substructure connectivity of $ G $, denoted by $ \kappa^s(G;\mathcal{H}) $, is the minimum number of a set of connected subgraphs in $ G $, whose removal disconnects $ G $ and each element in the set is isomorphic to a connected subgraph of $ \mathcal{H} $.  In this paper, we determine  $ \mathcal{H} $-structure connectivity and $ \mathcal{H} $-substructure connectivity of folded divide-and-swap cube $ FDSC_n $ for $ \mathcal{H}\in\{K_1, K_{1,1}, K_{1,m} (2\leq m \leq d+1) \} $ where $ n=2^d $.  We show that $\kappa(FDSC_n;K_1)=\kappa^s(FDSC_n;K_1)=d+2$, $\kappa(FDSC_n;K_{1,1})=\kappa^s(FDSC_n;K_{1,1})=d+1 $ for $ d\geq1 $ and  $\kappa(FDSC_n;K_{1,m})=\kappa^s(FDSC_n;K_{1,m})=\lfloor\frac{d}{2}\rfloor+1$ for $d\geq1 $ and $ 2\leq m \leq d+1$.

\end{abstract}

\bigskip


\bigskip


\section{Introduction}

In parallel computing, interconnection networks play a significant role.  An interconnection network can be modeled by a graph $G=(V(G), E(G))$ where $V(G)$ is the vertex set and $E(G)$ is the edge set.  In general, a vertex in $V(G)$ corresponds to a processor, and an edge in $E(G)$ corresponds to a communication link between two processors. The topological properties of interconnection networks have been studied thoroughly in the literature. 

For an interconnection network, it is vital to determine the fault tolerance of the system, since it reflects the resistance of the network against failures. Accordingly, various parameters have been defined and investigated extensively in order to measure the reliability of graphs. Among these, classical connectivity is one of the most important parameters to measure the reliability of the graph since it gives the minimum cost to disconnect it. More precisely, the connectivity of a graph $G$, denoted by $\kappa(G)$, is the minimum cardinality of a vertex set $S\subseteq V(G)$ such that $G-S$ is disconnected or trivial. This notion implicitly assumes that all the neighbors of a vertex may fail simultaneously, which is unlikely for large-scale networks. In 1983, Harary \cite{harary1983conditional} proposed conditional connectivity to overcome this shortcoming. The conditional connectivity, $\kappa(G, \rho) $  is the minimum cardinality of a vertex set $S \subseteq V(G) $, such that $G-S$ is disconnected and every component of it still has the property $\rho$. Motivated by this definition, several variants of this notion have been proposed and investigated extensively in literature \cite{boesch1986synthesis, esfahanian1989generalized, fabrega1994extraconnectivity, guo2023connectivity,hong2013extra, lee2020r3,  lin2015extra,zhou2017g}.  

Prior to 2016, most works on network reliability and fault tolerance focused on individual vertices becoming faulty. This approach implicitly assumes that the status of a vertex $v$ is independent of the status of its neighbors and it disregards the influence of the faulty vertex on its neighborhood. Vertices that are related may, however, affect each other, and those that are neighbors of a faulty vertex are more susceptible to being faulty later on. Since networks are increasingly integrated into chips in today's technology, considering the entire chip as faulty makes sense if any vertex on it becomes faulty. By this motivation, Lin \textit{et al.} \cite{lin2016structure} proposed the structure connectivity and substructure connectivity. Let $\mathcal{H}$ be a connected subgraph of $G$, and let $\mathcal{F}=\{F_1,F_2,\ldots,F_n\}$ be a set of subgraphs of $G$ where $F_i$ in $\mathcal{F}$ is isomorphic to $\mathcal{H}$ for $i\in\{1,2,\ldots,n\}$. Then $\mathcal{F}$ is a  $\mathcal{H}$-structure-cut if $G-\mathcal{F}$ is a disconnected or trivial graph. The minimum cardinality of all $\mathcal{H}$-structure-cuts of $G$ is the $\mathcal{H}$-structure connectivity of $G$, denoted by $\kappa(G;\mathcal{H})$.  Let $H$ be a connected subgraph of $G$, and let $\mathcal{F}=\{F_1,F_2,\ldots,F_n\}$ be a set of subgraphs of $G$ where $F_i$ in $\mathcal{F}$ is isomorphic to a connected subgraph of $\mathcal{H}$ for $i\in\{1,2,\ldots,n\}$.  Then $\mathcal{F}$ is a  $\mathcal{H}$-substructure-cut if $G-\mathcal{F}$ is a disconnected or trivial graph.  The minimum cardinality of all $\mathcal{H}$-substructure-cuts of $G$ is the $\mathcal{H}$-substructure connectivity of $G$, denoted by $\kappa^s(G;\mathcal{H})$. By the definitions, two subgraphs in an $\mathcal{H}$-structure-cut or an $\mathcal{H}$-substructure-cut are not necessarily disjoint.

The hypercube is a well-known interconnection network topology with several desirable properties such as symmetry, simple routing, maximal connectivity, and recursive structure. In literature, various variants of the classical hypercube have been proposed and received considerable attention. These hypercube variants have been investigated in terms of several reliability parameters including the structure and substructure connectivity \cite{ba2023star,li2019structure,lu2020structure,lv2018structure,pan2020structure,sabir2018structure,wang2023structure,zhou2021structureDSC}. 

In \cite{kim2019divide}, Kim \textit{et al.} introduced two novel hypercube variants namely divide-and-swap cube and folded divide-and-swap cube, with several nice hierarchical properties. 

A network's performance and effectiveness can be assessed by its properties such as its diameter, connectivity, fault tolerance, bisection width, broadcasting time, etc. \cite{akers1989group}. Kim \textit{et al.} \cite{kim2019divide} proposed $DSC_n$ and $FDSC_n$ to reduce the network cost defined by the product of degree and diameter. For the classical hypercube and its existing variations, the network cost is ${O}(n^2)$, whereas it is ${O}(n\log{}n)$ for $DSC_n$ and $FDSC_n$.  They also provided many properties and algorithms of these two new classes, including bisection width, Hamiltonicity, routing algorithm, one-to-all and all-to-all broadcasting algorithms. 

The fault tolerance of the divide-and-swap cube has been discussed in several papers. Ning \cite{ning2020connectivity} proved that the (edge) connectivity is equal to $d+1$ and the super (edge) connectivity is equal to $2d$ for $DSC_n$, where $n=2^d$ for $d\geq 1$. The super (edge) connectivity is a variant of (edge) connectivity that gives the minimum number of vertices (resp. edges) that need to be deleted to disconnect the graph without isolating a vertex. Later, Zhou \textit{et al.} \cite{zhou2021structureDSC} investigated the $\mathcal{H}$-structure connectivity and $\mathcal{H}$-substructure connectivity of $DSC_n$ for $ \mathcal{H} \in \{K_1, K_{1,1}, K_{1,m} \text{ } (2\leq m \leq d+1), C_4\}$. Zhou \textit{et al.} \cite{zhou2022reliabilityDSC} studied the $r$-component connectivity and diagnosability of $DSC_n$. Recently, Zhao and Chang \cite{zhao2023reliabilityDSC} determined the generalized $k$-connectivity of  $DSC_n$ for $k \in \{3,4\}$. 

The folded divide-and-swap cube is obtained from the divide-and-swap cube by adding an edge to each vertex to reduce the diameter slightly. In 2021, Chang \textit{et al.} \cite{chang2021constructing} showed that $FDSC_n$ is suitable as a candidate topology for data center networks and they provided a recursive construction of two completely independent spanning trees for $FDSC_n$. Zhao and Chang \cite{zhao2023connectivity} discussed the reliability of $FDSC_n$ and proved that the (edge) connectivity is equal to $d+2$. They also showed that the super connectivity is $2d$ and the super edge connectivity is $2d+2$. They also determined the generalized $3$-connectivity of $FDSC_n$. Recently, You \textit{et al.} \cite{you2023superspanning} investigated the super spanning connectivity of $FDSC_n$ and Xue \textit{et al.} \cite{xue2023gen4conn} provided an upper and a lower bound for the generalized $4$-connectivity of $FDSC_n$. Currently, no further research has been conducted on the reliability of $FDSC_n$.

In this paper, we continue to expand on the study of reliability in folded divide-and-swap cubes by considering structure connectivity and substructure connectivity. To be more precise, this study focuses on $ \mathcal{H} $-structure connectivity and $ \mathcal{H} $-substructure connectivity of  $ FDSC_n $ for $ \mathcal{H}\in\{K_1, K_{1,1}, K_{1,m} (2\leq m \leq d+1) \} $ where $ d\geq1 $ and $ n=2^d $.  

\section{Preliminaries}

In this section, we first introduce the notion used and then give the definition of the folded divide-and-swap cube. We also cite some known lemmas and then state two lemmas required in the proof of the main results. 

Let $G$ be an undirected graph with the vertex set $V(G)$ and the edge set $E(G)$.  For two vertices $u,v\in V(G)$, if $(u,v)\in E(G)$, then $u$ is adjacent to $v$ or $u$ is a neighbor of $v$. The neighborhood of $u \in V(G)$, denoted by $N_G(u)$, is the set of vertices adjacent to $u$ in $G$. The degree of $u$, denoted by $deg_G(u)$, is the cardinality of $N_G(u)$. If $deg_G(u)=r$ for each $u \in V(G)$, then $G$ is called $r$-regular. For any vertex set $S \subseteq V(G)$, let $G[S]$ denote the subgraph induced by $S$. We denote $[m]=\{1,2,\ldots,m\}$ and $N=[2^{\frac{n}{2}}]$ throughout the paper.

For a subgraph $H$ of a graph $G$, we use $G-H$ to denote the subgraph of $G$ induced by $V(G)-V(H)$.  For a set $\mathcal{F}=\{F_1, \dots, F_n\}$, where each $F_i$ is isomorphic to a connected subgraph of $G$, we use $G-\mathcal{F}$ to denote the subgraph of $G$ induced by $V(G)-V(F_1)-\ldots-V(F_n)$.

As a novel hypercube variant, the $n$-dimensional divide-and-swap cube $DSC_n$ was introduced by Kim \textit{et al.} \cite{kim2019divide} as follows. 

\begin{definition}\cite{kim2019divide}	
	For an integer $n=2^d$ and $d\geq1$, the n-dimensional divide-and-swap cube, denoted by $DSC_n$, is a graph with the vertex set $V(DSC_n) = \{0,1\}^n = \{u~\vert~ u=s_1s_2\ldots s_{n-1}s_n\textrm{ and } s_i \in \{0,1\} \textrm{ for } i\in [n]\}$. The label of $u$ can be divided into three parts, denoted as $u=s_1s_2\ldots s_{n-1}s_n=m_1m_2m_3$ such that 
	\[m_1=s_1s_2\ldots s_{\frac{n}{2^k}}, ~ m_2=s_{\frac{n}{2^k}+1}s_{\frac{n}{2^k}+2}\ldots s_{\frac{n}{2^{k-1}}}, ~m_3=s_{\frac{n}{2^{k-1}}+1}s_{\frac{n}{2^{k-1}}+2}\ldots s_n \\ \]
	where $1\leq k \leq log_{2}n=d$.
	If $k=1$, then $m_3$ is an empty string, that is, $m_1=s_1s_2\ldots s_{\frac{n}{2}}$ and $m_2=s_{\frac{n}{2}+1}s_{\frac{n}{2}+2}\ldots s_n$. A vertex $v$ is adjacent to $u$ in $DSC_n$ if it satisfies one of the following conditions:

	\begin{itemize}
		\item [(1)] $v=\overline{s}_1s_2s_3\cdots s_n$, where $\overline{s}_1$ is the complement of $s_1$. This type of edge is called an $e(1)$-edge. 
		\item [(2)] $v=\overline{m}_1\overline{m}_2m_3$ if $m_1=m_2$; and $v=m_2m_1m_3$  otherwise. This type of edge is called an $ e(\frac{2n}{2^k}) $-edge. 		
	\end{itemize}
\end{definition}

The folded hypercube is one of the well-known variants of the classical hypercube. With similar motivation, Kim \textit{et al.} \cite{kim2019divide} proposed the folded divide-and-swap cube obtained by adding an edge to each vertex. 
  
\begin{definition}\cite{kim2019divide}\label{defFDSCn}
	For $n=2^d$ and $d\geq1$, the n-dimensional folded divide-and-swap cube, denoted by $FDSC_n$, is obtained from $DSC_n$ by adding an edge to each vertex as follows: $V(FDSC_n) = V(DSC_n)$ and $E(FDSC_n) = E(DSC_n) \cup E_f$, where 
	$E_f = \{ (u,v) ~\vert~ u=u_1u_2u_3\dots u_n \text{ and } v=u_1\Bar{u}_2u_3\dots u_n\}$ and each $(u,v) \in E_f$ is called an $e(f)$-edge.
\end{definition}

For any vertex $u =u_1u_2u_3\dots u_n$ in $FDSC_n$, we use the notation $u_f$ for the vertex $u_1\Bar{u}_2u_3\dots u_n$. In Figure 1, $FDSC_2$ and $FDSC_4$ are given. 

\begin{figure}[h]
	\centering
	\includegraphics[width=0.8\textwidth]{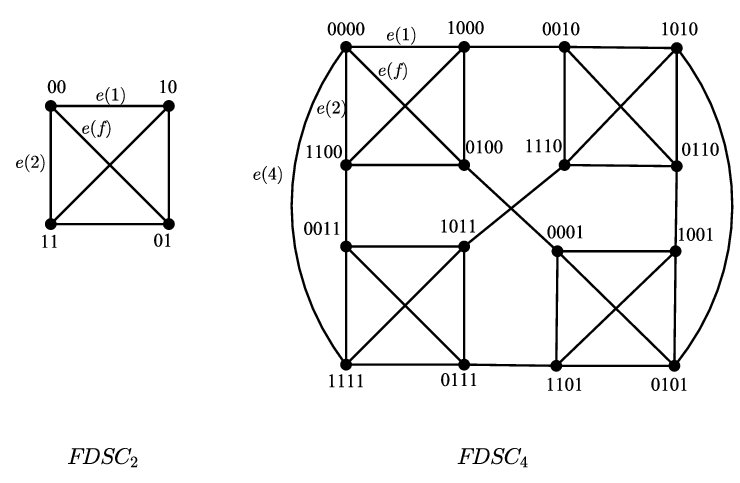}
	\caption{The folded divide-and-swap cubes $FDSC_2$ and $FDSC_4$ }
\end{figure}

Each subgraph $FDSC_{\frac{n}{2}}$ in $FDSC_n$ is referred as a module. The address of an arbitrary vertex  $ u=s_1s_2\ldots s_{\frac{n}{2}}s_{\frac{n}{2}+1}s_{\frac{n}{2}+2} \ldots s_n$ in a module can be represented by $u=A_iB_i$, where $A_i=s_1s_2\ldots s_{\frac{n}{2}}$ is the address of the vertex inside the module and $B_i=s_{\frac{n}{2}+1}s_{\frac{n}{2}+2} \ldots s_n$ is the address of the module. A module with address $B_i$ is denoted by $G_i$. An edge inside a module is called an interior edge, while an edge connecting two distinct modules is called a cross edge. That is, the edges $ e(i) $ for $ i\in\{1,\frac{2n}{2^d}, \frac{2n}{2^{d-1}}, \dots, \frac{2n}{2^4}, \frac{2n}{2^3}, \frac{2n}{2^2}\} $ and $ e(f) $ are interior edges. The edge $ e(n) $ is a cross edge. Throughout the paper,  we call an edge $ (u, v)\in E(FDSC_n) $ fault-free if the end vertices $ u $ and $ v $ are both fault-free.

If $ (u, v) $ is an $ e(1) $-edge, then we refer the vertex $ v $ as the $1$-neighbor of vertex $ u $, and we also denote $ v =u_1 $; if $ (u, v) $ is an $ e(\frac{2n}{2^k}) $-edge for $ 2\leq k\leq d $, then we refer $ v $ as the $ k $-neighbor (interior neighbor) of vertex $ u $, denoted by $ u_k $; if $ (u, v) $ is an $ e(\frac{2n}{2^k}) $-edge where $ k=1 $, then we refer $ v $ as the ($ d+1 $)-neighbor (external neighbor), namely $ v =u_{d+1} $.

\begin{lemma} \cite{kim2019divide} \label{knownproportiesFDSC}
	For any $n=2^d$ and $d\geq 1$, we have the following results:
	\begin{itemize}
		\item [(1)] $FDSC_n$ is ($d+2$)-regular.
		\item [(2)] $|V(FDSC_n)|=2^n$ and $|E(FDSC_n)|=2^{n-1}(d+2)$.	
		\item [(3)] $FDSC_n$ can be decomposed into $2^\frac{n}{2}$ $FDSC_{\frac{n}{2}}$.
		\item [(4)] The minimum length of a cycle in  $FDSC_n$ is $3$.
		\item [(5)]If we represent each module $FDSC_{\frac{n}{2}}$ in $FDSC_n$  with a super vertex, the obtained graph is a complete graph $K_{2^{\frac{n}{2}}}$.
	\end{itemize}
\end{lemma}

\begin{lemma}\cite{kim2019divide} \label{externalneighbors}
	Let $B_i$ be the address of each module $G_i$ in $FDSC_n$ for all $i\in N$. If $B_i=\overline{B_j}$ for $i,j \in N$ and $i\neq j$, then two modules $G_i$ and $G_j$ with addresses $B_i$ and $B_j$ are connected by two edges $(B_iB_i,B_jB_j)$ and $(B_jB_i,B_iB_j)$; otherwise, $G_i$ and $G_j$ are connected by one edge $(B_jB_i,B_iB_j)$.
\end{lemma}

\begin{lemma}\cite{you2023superspanning} \label{externalneighbors2}
	Let $G_1,G_2,\ldots,G_{2^{\frac{n}{2}}}$ be the $2^{\frac{n}{2}}$ modules of $FDSC_n$
	each of which has $2^{\frac{n}{2}}$ vertices. Let $u=B_iB_i$ in module $ G_i$ with $1\leq i \leq 2^{\frac{n}{2}}$. Then each vertex in $V(G_i)-\{u\}$ is connected to a different module $G_j$ with an $e(n)$-edge where $1\leq i\neq j \leq 2^{\frac{n}{2}}$. Let $v=\overline{B_i}B_i$ in $ G_i$. Then $u $ and $v $ are connected to two distinct vertices from $G_{k}$ through $e(n)$-edges, where the address of $G_k$ is  $\overline{B_i}$ with $i\neq j \neq k$. 
\end{lemma}

Let us consider Lemma \ref{externalneighbors}. The vertices $B_jB_j=\overline{B}_i \overline{B}_i$ and $B_iB_j=B_i\overline{B}_i$ are the external neighbors of $B_iB_i$ and $B_jB_i=\overline{B_i}B_i$, respectively. Thus, by Lemma \ref{externalneighbors} and  Lemma \ref{externalneighbors2}, we readily have the following lemma.

\begin{lemma}\label{proportiesFDSC}
	For any module $G_i$ in $FDSC_n$  with address $B_i$ ($i\in N$),  the followings are true.
	\begin{itemize}
		\item [(1)] Each vertex in $G_i$ has exactly one external neighbor in some module $G_j$ for $j\neq i$.
		\item [(2)] Vertices $B_iB_i$ and $\overline{B_i}B_i$ have different external neighbors $\overline{B}_i\overline{B}_i$ and $B_i\overline{B_i}$ respectively in the same module with address $\overline{B_i}$.
		\item [(3)] The external neighbors of any two vertices in $V(G_i)-\{B_iB_i, \overline{B_i}B_i\}$ are in different modules.
		\item [(4)] The number of cross edges between two different modules is either 1 or 2.
	\end{itemize}
\end{lemma}

We prove the following useful lemma. 
\begin{lemma} \label{nocommonneighbor}
	The vertices $u=B_iB_i$ (resp. $u=\overline{B_i}~\overline{B_i})$ and $v=\overline{B_i}B_i$ (resp. $v=B_i\overline{B_i})$ do not have any common neighbors in the module $G_i$ with address $B_i$ (resp.  $\overline{B_i}$).
\end{lemma}

\begin{proof}
	Without loss of generality, let $u=B_iB_i$ and $v=\overline{B_i}B_i$. Let $N(u)= \{u_1,u_2,\dots,u_{d+1},u_f\}$ and $N(v)= \{v_1,v_2,\dots,v_{d+1},v_f\}$. It is enough to prove that $u_j\neq v_k$ for any $j,k \in\{1,2,\ldots,d+1,f\}$. For the external neighbors of $u$ and $v$, we have $u_{d+1}\neq v_{d+1}$ by Lemma \ref{proportiesFDSC}. Thus, we consider interior neighbors of $u$ and $v$. Let $B_i=CD$, where $C$ and $D$ are two binary strings of length $\frac{n}{4}$. Thus, $u=CDB_i$ and $v=\overline{C}~\overline{D}B_i$.  By the definition of $FDSC_n$, the rightmost $\frac{3n}{4}$-bit binary strings of $u_j$ and $v_k$ for some $j,k \in\{1,2,\ldots,d,f\}-\{2\}$ are $DB_i$ and $\overline{D}B_i$, respectively. Since  $D\neq\overline{D}$, we have $u_j\neq v_k$ for $j,k\in\{1,2,\ldots,d,f\}-\{2\}$.
	
	Let $j=k=2$. If $C=D$, then $u_2=\overline{C} ~\overline{D}B_i=v$ and $v_2=CDB_i=u$. That is, $u_2\neq v_2$. If $C\neq D $, then $u_2= DCB_i$ and $v_2= \overline{D}~\overline{C}B_i$. Since  $DC\neq \overline{D}~\overline{C}$, we have $u_2\neq v_2$. Thus, $u$ and $v$ do not have any common neighbors.	
\end{proof}

\section{Main Results}

In this section, we present our main results on the structure connectivity of folded divide-and-swap cube. Before discussing the proofs of our results, it would be worth noting that $\kappa(G;\mathcal{H}) \geq \kappa^s (G;\mathcal{H})$ for any graph $G$ and a connected subgraph $\mathcal{H}\subseteq G$  \cite{li2018structure}. This fact will be used later in our proofs.

We first consider $ K_1 $-structure connectivity and $ K_1 $-substructure connectivity of $ FDSC_n $.  For any graph $G$, we know that $\kappa(G)=\kappa(G;K_1)=\kappa^s(G;K_1)$  \cite{lin2016structure}. Since Zhao \textit{et al.} \cite{zhao2023connectivity} recently showed that $\kappa(FDSC_n) = d+2$, we have the following theorem.

\begin{theorem} \label{thm:K1} For $ d\geq1 $ and $ n=2^d $,	$\kappa(FDSC_n;K_1)=\kappa^s(FDSC_n;K_1)=d+2$.
\end{theorem}

We now present the following two useful lemmas to prove our first main result on $K_{1,1}$-structure connectivity and $K_{1,1}$-substructure connectivity of $FDSC_n$.

\begin{lemma} \label{lemma1-K11lowerbound}
	Let $ A_1 $ be a subset of $ \{x~|~ x\in V(FDSC_n)\} $ and $ A_2 $ be a subset of $ \{\{y,z\}~|~ (y,z)\in E(FDSC_n)\} $ with $ |A_1|+|A_2|\leq d $ and $ |A_2|\leq d-1 $ for $ d\geq $3. Then $ FDSC_n-(A_1\cup A_2) $ is connected.
\end{lemma}

\begin{proof}
	Suppose to the contrary that $ FDSC_n-(A_1\cup A_2) $ is disconnected. Let $ C $ be the smallest component of $ FDSC_n-(A_1\cup A_2) $. We know that the super connectivity of $FDSC_n$ is equal to $2d$. That is, at least $2d$ vertices need to be deleted from $FDSC_n$ to disconnect it without isolating a vertex. Since $ |V(A_1\cup A_2)|\leq 2d-1 $ and $ FDSC_n-(A_1\cup A_2) $ is disconnected, we have $ |V(C)|=1 $. Let $ V(C)=\{u\} $. For two vertices $ v,w\in N_{FDSC_n}(u) $, we know that $ |N_{FDSC_n}(v)\cap N_{FDSC_n}(w)-\{u\}|\leq 1 $ by the definition of $FDSC_n$. Note that there are three vertices $ p,q,r\in N_{FDSC_n}(u)$  such that the subgraph induced by $\{p,q,r\}$  is isomorphic to  $ K_3 $ and $ N_{FDSC_n}(u)-\{p,q,r\} $ is an independent set. Thus, 
	$ |N_{FDSC_n}(u)\cap V(A_1)|\leq|A_1| $  and  $ |N_{FDSC_n}(u)\cap V(A_2)|\leq|A_2|+1 $. Since
	\begin{align*}
		|N_{FDSC_n}(u)\cap V(A_1\cup A_2)|&\leq |N_{FDSC_n}(u)\cap V(A_1)|+|N_{FDSC_n}(u)\cap V(A_2)|\\ &\leq |A_1|+|A_2|+1\\&\leq d+1\\&<d+2=|N_{FDSC_n}(u)|,
	\end{align*}
	there exists a vertex $ t\in N(u)-V(A_1\cup A_2) $. This contradicts that $ |V(C)|=1 $. Thus, $ FDSC_n-(A_1\cup A_2) $ is connected.
\end{proof}

\begin{lemma} \label{lowerboundK11}
	For $ d\geq $3, let $ A_1 $ be a subset of $ \{x~|~ x\in V(FDSC_n)\} $ and $ A_2 $ be a subset of $ \{\{y,z\}~|~ (y,z)\in E(FDSC_n)\} $ with $ |A_1|+|A_2|\leq d $. Then $ FDSC_n-(A_1\cup A_2) $ is connected.
\end{lemma}

\begin{proof} We prove this statement by induction on $n$. It is easy to check that the statement is true for $FDSC_8$. Assume that the statement holds for $ FDSC_{\frac{n}{2}} $. 
	
	Note that by Lemma \ref{lemma1-K11lowerbound}, the remaining graph $ FDSC_n-(A_1\cup A_2) $ is connected if $ |A_2|\leq d-1 $. Thus, it is enough to consider the case when $ |A_2|=d $ and $ |A_1|=0 $ to complete the proof.\\
	
	Considering a module  $G_i $ of $FDSC_n$, we first let
	\[{P_i=\{x~|~\{x,y\}\in A_2, x\in V(G_i), y\notin V(G_i)\}},\] \[ Q_i=\{\{y,z\}~|~ \{y,z\}\in A_2, y\in V(G_i), z\in V(G_i) \}\]
	and let \[ R_i=\{\{p_1,p_2\}~| ~\{p_1,p_2\}\in A_2 ~\text{and}~ p_1\in P_i\} \cup \{\{q_1,q_2\}~| ~\{q_1,q_2\}\in Q_i\}\]  for $ i\in N $.
	We obviously have $ |P_i|+|Q_i|\leq |A_1|+|A_2|=d $ for each $i \in N$. 
	
	If $  |P_i|+|Q_i|=0$ for any  $ i\in N $, then the module $G_i$ is called intact in $FDSC_n - (A_1 \cup A_2)$. By Lemma \ref{proportiesFDSC} (4), all the intact modules of $FDSC_n - (A_1 \cup A_2)$ are contained in the same component, say $C$, of the remaining graph. 
	\\
	\noindent We need to consider the following two cases. 
	
	\noindent\textbf{Case 1.} Let $ |P_i|+|Q_i|\leq d-2$ for each $ i\in N$.\\
	Consider any module $G_k$ with $ |P_k|+|Q_k|>0$. Note that $ G_k- (P_k\cup Q_k) $ is connected by the induction hypothesis. We know that there are at most $2d$ modules that are not intact in $FDSC_n - (A_1 \cup A_2)$. Since 
	\[	
	2^\frac{n}{2}-|V(P_k\cup Q_k)|\geq2^\frac{n}{2}-2(d-2)>2d-1
	\]
	when $d\geq 3$ for each $G_k$, there exists an edge joining a vertex from $ G_k- (P_k\cup Q_k) $ to a vertex from an intact module contained in $C$. Thus,  the remaining graph $ FDSC_n-(A_1\cup A_2) $ is connected.
	
	\noindent\textbf{Case 2.} Let $ |P_i|+|Q_i|\geq d-1$ for some $ i\in N$.\\
	Without loss of generality, assume that $ |P_1|+|Q_1|\geq d-1$ and $ |P_1|+|Q_1|= \max\{|P_i|+|Q_i|~|~i\in N\} $. There are two cases to consider.
	
	\textbf{Case 2.1.} If $ |P_1|+|Q_1|=d $, then we have the following observations:
	\begin{itemize}
		\item[\textbf{(F1)}] For every vertex $ z\in V(A_2) - V(P_1\cup Q_1) $, there exists a vertex $ y\in P_1 $ such that $ \{y,z\}\in A_2 $. That is, $|Q_i|=0 $ and $R_i\cap R_j= \emptyset $ for each $i,j\in N-\{1\}$, where $  i\neq j$. 
		\item[\textbf{(F2)}] 
		
		By Lemma \ref{proportiesFDSC} (2), the number of vertices in $G_1$ having their external neighbors in the same module is exactly two. Without loss of generality, assume that these vertices have their external neighbors in $G_2$. Note that $|Q_2|=0$ by (F1), thus we have $ |P_2|\leq 2 $.
		
		\item[\textbf{(F3)}]  For each $ i\in \{3,\ldots,|P_1|+1\}$, since  $|Q_i|=0$ by (F1), we have $ |P_i|\leq 1 $. 
		
		\item[\textbf{(F4)}] $ |P_i|+|Q_i|=0 $  for each $ i\in \{|P_1|+2,\ldots, 2^\frac{n}{2}\}$.		
	\end{itemize} 
	
	By the induction hypothesis, $ G_i- (P_i\cup Q_i) $ is connected for each $ i\in N-\{1\} $ and has at least $2^\frac{n}{2}-2$ vertices. There are at most $d+1$ modules that are not intact in $ FDSC_n-(A_1\cup A_2) $. Since $2^\frac{n}{2}-2 > d $ for $d\geq 3$, there exists an edge joining a vertex from $ G_i- (P_i\cup Q_i) $ to a vertex from an intact module contained in $C$ for each $i\in N-\{1\}$. That is, $ FDSC_n-G_1-(A_1\cup A_2) $ is connected.

	Consider a vertex $u \in G_1-(P_1\cup Q_1)$. Let $u_{d+1}\in V(A_1\cup A_2)$. Since $|Q_i|=0$ for each $ i\in N-\{1\} $, we have $u_{d+1}\in A_1$. This contradicts the fact that $|A_1|=0$. We then let $u_{d+1}\notin V(A_1\cup A_2)$. It follows that the fault-free edge $(u,u_{d+1})$ connects the component containing $u$ to $ FDSC_n-G_1-(A_1\cup A_2) $. Thus, $ FDSC_n-(A_1\cup A_2) $ is connected.

	\textbf{Case 2.2.} Let $ |P_1|+|Q_1|=d-1$.  Since $ |P_1|+|Q_1|=d-1$ and $|A_2|=d$, we have $|A_2|-|R_1|=1$.  Thus, either there exists exactly one module $G_i$ where $|Q_i|=1$ for $i\in N-\{1\}$ or there exist exactly two modules $G_i$ and $G_j$ such that $R_i\cap R_j\neq \emptyset $ for $i,j\in N-\{1\}$ and $ i\neq j$. Since $ |P_i|+|Q_i| \leq d-1$ for each $i\in N$, it is clear that $ G_i- (P_i\cup Q_i) $ is connected by the induction hypothesis. 
	
	For any module $G_i$ with $ |P_i|+|Q_i| \neq 0 $, there are at least $2^\frac{n}{2} - 2(d-1)$ vertices in  $ G_i- (P_i\cup Q_i) $. We also know that there are at most $d+2$ modules that are not intact. For the modules which are not intact, we consider the external neighbors of the remaining vertices. Since
	\[ 2^\frac{n}{2} - 2(d-1) > d+1\] when $d \geq 3$ for each $G_i$, there exists an edge joining a vertex from  $ G_i- (P_i\cup Q_i) $ to a vertex from an intact module contained in $C$. Thus,  the remaining graph $ FDSC_n-(A_1\cup A_2) $ is connected.
\end{proof}

It is easy to check that $\kappa(FDSC_n;K_{1,1})=\kappa^s(FDSC_n;K_{1,1})=2$ for $d\in \{1,2\}$ and $ n=2^d $. In the following theorem, we determine $\kappa(FDSC_n;K_{1,1}) $ and $\kappa^s(FDSC_n;K_{1,1})$ for $d\geq 3$ and $n=2^d$.

\begin{theorem} \label{thm:K11} For $ d\geq3 $ and $ n=2^d $,
	$\kappa(FDSC_n;K_{1,1})=\kappa^s(FDSC_n;K_{1,1})=d+1$.
\end{theorem}	

\begin{proof}
	Consider a vertex $ u $ from any module, say $ G_i $, of $ FDSC_n $. Let $N_{FDSC_n}(u) =\{u_1,u_2\ldots,\\u_{d+1},u_f\} $. The induced subgraph $ FDSC_n[{u_j,(u_j)_1}] $	will be denoted by $ F_j $ for $ j\in\{2,\ldots,d+1\} $, where $(u_d)_1=u_f$. Let $ F_1 $ denote $ FDSC_n[{u_1,(u_1)_{d+1}}] $. Obviously, $ F_j\cong K_{1,1} $ for each $ j\in\{1,\ldots,d+1\} $ (see Figure \ref{K11}). Let $ \mathcal{F}=\{F_1,F_2,\ldots,F_{d+1}\} $. Note that $ FDSC_n-\mathcal{F} $ is disconnected. Thus, $\kappa(FDSC_n;K_{1,1})\leq d+1$.
	
	By Lemma \ref{lowerboundK11}, we know that any $K_{1,1}$-substructure-cut has more than $d$ elements. That is,  $\kappa^s(FDSC_n;K_{1,1})\geq d+1$. Note that  $\kappa(FDSC_n;K_{1,1}) \geq \kappa^s(FDSC_n;K_{1,1})$ and this finishes the proof.
	\begin{figure}[h]
		\centering
		\includegraphics[width=0.8\textwidth]{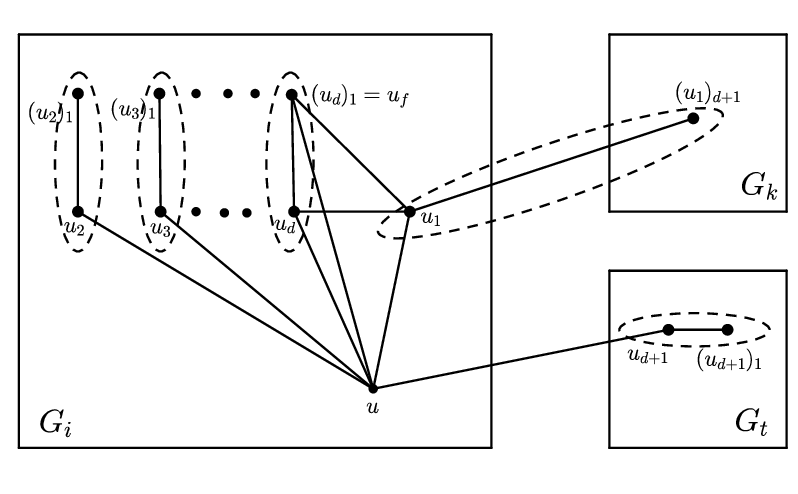}
		\caption{A $K_{1,1}$-structure-cut of $FDSC_n$}
		\label{K11}
	\end{figure}
\end{proof}

In the rest of this section, we consider $K_{1,m}$-structure connectivity and $K_{1,m}$-substructure connectivity of $FDSC_n$ where $ 2\leq m \leq d+2 $. We first prove the following two lemmas.

\begin{lemma} \label{lowerboundK1m}For $ d\geq1 $ and $ n=2^d $,
	$\kappa^s(FDSC_n;K_{1,m})\geq \lfloor\frac{d}{2}\rfloor+1$ where $ 2\leq m \leq d+2 $.
\end{lemma}

\begin{proof}
	It can be readily checked that the statement holds when $d$ is either $1$ or $2$. In the sequel, we let $d\geq 3$ and proceed the proof by induction on $n$. To establish the base case, it is easy to verify that $ \kappa^s(FDSC_8;K_{1,m})\geq 2$ for $ 2\leq m \leq 5 $.  We now assume that the statement holds for $ FDSC_{\frac{n}{2}}$.

	Suppose to the contrary that there exists a $ K_{1,m} $-substructure-cut  of $ FDSC_n $, say $ \mathcal{F} $, such that $ |\mathcal{F}|\leq\lfloor\frac{d}{2}\rfloor $. Let $\mathcal{F} = \{F_1, F_2,\dots,F_x\}$ be a set of subgraphs of $ FDSC_n $ where $ F_i $ is isomorphic to a connected subgraph of $ K_{1,m} $ and let $ x\leq \lfloor\frac{d}{2}\rfloor  $.

	For each $i \in N$, we let $\mathcal{F}^i=\bigcup_{F_j\in\mathcal{F} } F_j \cap FDSC_{\frac{n}{2}}^i$. It is clear that any element of $ \mathcal{F}^i $ is isomorphic to a connected subgraph of $ K_{1,m} $. Note that $ |\mathcal{F}^i|\leq |\mathcal{F}|\leq \lfloor\frac{d}{2}\rfloor$ for each $ i\in N $ by  Definition \ref{defFDSCn}. 
	
	Let $ T=\{i~\vert~ \mathcal{F}\cap G_i\neq \emptyset\} $ and let $ G_T=\cup_{i\in T}G_i$. Throughout this proof, we call a module $G_i$ intact if $G_i \cap \mathcal{F} = \emptyset$. Note here that $ G_T$ is the union of the modules that are not intact in the remaining graph $FDSC_n - \mathcal{F}$. The graph $ FDSC_n-G_T $ is connected by Lemma \ref{proportiesFDSC} (4). 
	
	There are two cases to consider. 
		
	\noindent\textbf{Case 1.} Suppose that $ G_i-\mathcal{F}^i $ is connected for every $ i\in N $.
	
	Since $|\mathcal{F}|\leq \lfloor\frac{d}{2}\rfloor$, there are at most $2\lfloor\frac{d}{2}\rfloor$ modules that are not intact by Lemma \ref{proportiesFDSC}. We now consider the external neighbors of vertices of $ G_i-\mathcal{F}^i $ for any $i\in T$. Note that
	\begin{align*}
		|V(G_i-\mathcal{F}^i)|&\geq |V(G_i)|-|V(K_{1,m})|\times|\mathcal{F}^i|\\& \geq 2^\frac{n}{2}-(m+1)\lfloor\frac{d}{2}\rfloor\\& >2\lfloor\frac{d}{2}\rfloor-1	
	\end{align*}
	for each $i\in T$, where $ d\geq3 $. Thus, there exists an edge joining a vertex from $ G_i-\mathcal{F}^i $ to a vertex from  $ FDSC_n-G_T $ for each $i\in T$.  Hence,  $ FDSC_n-\mathcal{F} $ is connected, a contradiction.

	\noindent\textbf{Case 2.} Suppose that $ G_i-\mathcal{F}^i $ is disconnected for some $ i\in N $, say $i=1$. 
	
	By the induction hypothesis, $ |\mathcal{F}^1|\geq \lfloor\frac{d-1}{2}\rfloor +1 $. Note that  $ |\mathcal{F}|\leq  \lfloor\frac{d}{2}\rfloor$ by the assumption and  $ \lfloor\frac{d-1}{2}\rfloor+1 \leq \lfloor\frac{d}{2}\rfloor$ holds only if $d$ is even. In the sequel of the proof, we assume that $d$ is even and thus  $|\mathcal{F}^1|= |\mathcal{F}|=\frac{d}{2}$, where $d\geq 4$. It follows that $F_i \cap G_1 \neq \emptyset $ for each $F_i \in \mathcal{F}$ and there are at most $ \frac{d}{2} $ modules which are not intact except the module $ G_1 $. By Lemma \ref{proportiesFDSC}, $|\mathcal{F}^i|$ is either 1 or 2 for $ i\in T-\{1\} $. We also know that there exists at most one module satisfying $|\mathcal{F}^i| = 2$  where $ i\in T-\{1\} $. Without loss of generality, we assume that $|\mathcal{F}^2|$ is either 1 or 2, and thus $|\mathcal{F}^i| =1$ for $ i\in T-\{1,2\} $. It is worth noting that there are at most $2m$ faulty vertices in $G_i$ for any $i\in T-\{1\}$. That is, $|V(\mathcal{F}^i)| \leq 2m$ for any $i\in T-\{1\}$.  
	
	Let us now consider the connectedness of each module in the remaining graph $ FDSC_n-\mathcal{F} $. For each $ i\in T-\{1\}$, we know that  $|\mathcal{F}^i|\leq 2 $. On the other hand, $\kappa^s(FDSC_{\frac{n}{2}};K_{1,m})\geq \lfloor\frac{d-1}{2}\rfloor+1$ by the induction hypothesis and $\lfloor\frac{d-1}{2}\rfloor+1 > 2 $ when $d>4$. Thus, $ G_i-\mathcal{F}^i $ is connected for each $ i\in T-\{1\}$ when $d \geq 6$.
	
	If $d \geq 6$, then
	\begin{align*}
		|V(G_i-\mathcal{F}^i)|\geq 2^\frac{n}{2}-2m >\frac{d}{2}	
	\end{align*}
	for each $i\in T-\{1\}$. Thus, there exists an edge joining a vertex from  $ G_i-\mathcal{F}^i $ to a vertex from  $ FDSC_n-G_T $ for each $i\in T-\{1\}$. That is,  $ FDSC_n-G_1-\mathcal{F} $ is connected.
		
	Let $d=4$. If $|\mathcal{F}^2|=1$, then   $G_i-\mathcal{F}^i $ is connected for each $ i\in T-\{1\}$ and it is easy to see that  $ FDSC_n-G_1-\mathcal{F} $ is connected. If $|\mathcal{F}^2|=2$, then $G_2-\mathcal{F}^2$ may or may not be connected. Note that the modules except $G_1$ and $G_2$ are intact since $|\mathcal{F}|\leq2$. That is, $FDSC_n-G_1-G_2$ is connected. Consider any vertex $u$ in $G_2-\mathcal{F}^2$. The external neighbor $u_{d+1}$ of $u$ is in $FDSC_n-G_1-G_2$. There exists  a fault-free edge $(u, u_{d+1})$ between $G_2-\mathcal{F}^2$ and $FDSC_n-G_1-G_2$. Thus, $ FDSC_n-G_1-\mathcal{F} $ is connected.
	
	We now consider the vertices in $G_1$ to complete the proof. Let  $ u=B_1B_1 $ and $ v=\overline{B_1}B_1 $ be two vertices in $ G_1 $, where $ B_1 $ is the address of the module $ G_1 $. Note that $ u_{d+1} $ and $  v_{d+1} $ are in the same module, say $ G_2 $ by Lemma \ref{proportiesFDSC} (2). Let $ C $ be a component in $ G_1-\mathcal{F}^1 $. We treat the cases $ |V(C)|>1 $ and $ |V(C)|=1 $ separately.
	
	\begin{itemize}
		\item [$ (1) $]
		Let $ |V(C)|>1 $. We then have three possibilities:
		\begin{itemize}
			\item [$ (a) $]Both $ u$ and $v$ are in $ C $.
			
			If there exists a vertex $w$ in $ C $ whose external neighbor $w_{d+1}$ is not a vertex of $ \mathcal{F }$, then $ C $ is connected to $ FDSC_n-G_1-\mathcal{F} $ with the fault-free edge  $ (w,w_{d+1}) $. Thus, $ u_{d+1} $ and $  v_{d+1} $ are both in $ V(\mathcal{F})$.
			
			\begin{itemize}
				\item[] Let $ u_{d+1} $ and $  v_{d+1} $ be in the same faulty subset $ F_i\in \mathcal{F} $. By Lemma \ref{nocommonneighbor}, both of $ u_{d+1} $ and $  v_{d+1} $ cannot belong to the set of leaves of $ F_i $. Without loss of generality, assume that $ u_{d+1} $ is the center of $F_i$ and $  v_{d+1} $ is a leaf of $F_i$. By Lemma \ref{proportiesFDSC}, $u$ is the only external neighbor of $ u_{d+1} $. Thus, all the vertices of $F_i$ are in $G_2$. Moreover, any neighbor of $ u_{d+1} $ except $ v_{d+1} $ cannot have an external neighbor in $ G_1 $ by Lemma \ref{proportiesFDSC} (see Figure \ref{K1mrenkli} (1)). Thus, $ F_i\cap G_1=\emptyset $ and $ F_i\cap G_2\neq \emptyset $, this contradicts that $ |\mathcal{F}^1|= |\mathcal{F}|=\frac{d}{2}$. 
				\begin{figure}[h]
					\centering
					\includegraphics[width=0.85\textwidth]{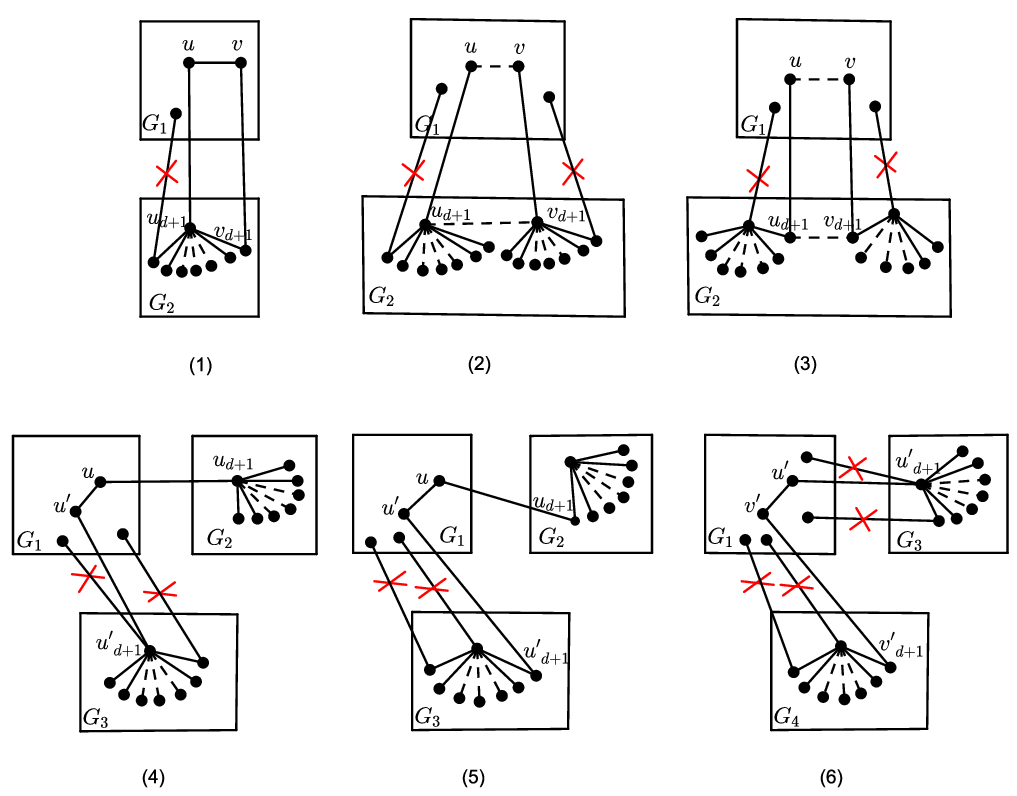}
					\caption{Explanation of Lemma \ref{lowerboundK1m}}
					\label{K1mrenkli}
				\end{figure}		
				\item[]   Let $ u_{d+1} $ and $  v_{d+1} $ be in different faulty subsets $ F_i $ and $ F_j $, respectively, where $ F_i, F_j\in \mathcal{F} $. We consider the following two cases. 
				\begin{itemize}
					\item[(i)] If $ u_{d+1} $ is the center of $ F_i $, then  all the vertices of $ F_i $ are in $ G_2 $ by Lemma \ref{proportiesFDSC}. Moreover, the neighbors of $ u_{d+1} $, except $ v_{d+1} $ (if they are adjacent), cannot have an external neighbor in $ G_1 $ by Lemma \ref{proportiesFDSC} (see Figure \ref{K1mrenkli} (2) for one of the two possible roles of $ v_{d+1} $ in $ F_j $). 
					\item[(ii)] If $ u_{d+1} $ is a leaf in $ F_i $, then the other vertices of $ F_i $  cannot be in $ G_1 $ by Lemma \ref{proportiesFDSC}. Moreover, any neighbor of the center of $ F_i $  cannot have an external neighbor in $ G_1 $ by Lemma \ref{proportiesFDSC} (see Figure \ref{K1mrenkli} (3) for one of the two possible roles of $v_{d+1} $ in $ F_j $). 
				\end{itemize}	
				Thus, $ F_i\cap G_1=\emptyset $ and $ F_i\cap G_2\neq \emptyset $ in both cases, this contradicts that $ |\mathcal{F}^1|= |\mathcal{F}|=\frac{d}{2}$. 				
				\end{itemize}

			\item [$ (b) $] Exactly one of $u$ and $v$ is in $C$. 
			
			We assume without loss of generality that $u\in C$. Since $ |V(C)|>1 $, there exists another vertex $ u' $ in $ C $ which is adjacent to $ u $. If there exists a vertex $ w $ in $ C $ such that $ w_{d+1}$ is not a vertex of $\mathcal{F}$, then $ C $ is connected to $ FDSC_n-G_1-\mathcal{F} $ with the fault-free edge $ (w,w_{d+1}) $. Thus, $ u_{d+1} $ and $ u'_{d+1} $ are both in $ V(\mathcal{F}) $. Moreover, by Lemma \ref{proportiesFDSC}, $ u_{d+1} $ and $ u'_{d+1} $ are in different modules, say $ G_2 $ and $ G_3 $, and different faulty subsets, say $ F_i $ and $ F_j $, respectively. 
			If we consider the cases where $ u'_{d+1} $ is the center of $ F_j $ (see Figure \ref{K1mrenkli} (4) for one of the two possible roles of $u_{d+1}$ in $ F_i $) and $ u'_{d+1} $ is a leaf of $ F_j $ (see Figure \ref{K1mrenkli} (5) for one of the two possible roles of $u_{d+1}$ in $ F_i $) separately, then we see in both cases that no vertex of $ F_j $ is in $ G_1 $ by Lemma \ref{proportiesFDSC}. In the former, any neighbor of $ u'_{d+1} $ cannot have an external neighbor in $ G_1 $, while in the latter, any neighbor of the center of $ F_j $, except $ u'_{d+1} $,  cannot have an external neighbor in $ G_1 $ by Lemma \ref{proportiesFDSC}. 
			Thus, $ F_j\cap G_1=\emptyset $ and $ F_j\cap G_3\neq\emptyset $ in both cases, this contradicts that $ |\mathcal{F}^1|= |\mathcal{F}|=\frac{d}{2}$.

			\item [$ (c) $]  Neither $ u$  nor $ v$  is in $ C$.
			
			There exist two adjacent vertices $ u'$ and $ v'$ in $ C $ different from $ u=B_1B_1 $ and $ v= \overline{B_1}B_1 $. If there exists a vertex $ w $ in $ C $ such that $ w_{d+1}$ is not a vertex of $\mathcal{F}$, then $ C $ is connected to $ FDSC_n-G_1-\mathcal{F} $ with the fault-free edge $ (w,w_{d+1}) $. Thus, $ u'_{d+1} $ and $ v'_{d+1} $ are  both in $ V(\mathcal{F}) $. Moreover, by Lemma \ref{proportiesFDSC}, $ u'_{d+1} $ and $ v'_{d+1} $ are in different modules, say $G_3$ and $G_4$, and different faulty subsets, say $ F_i $ and $ F_j $, respectively. Note that $ u'_{d+1} $ is either the center or a leaf of $ F_i $, whereas $ v'_{d+1} $  is either the center or a leaf of $ F_j $.  In each of these four possibilities, no vertices of $ F_i $ and $ F_j $ are in $ G_1 $, by Lemma \ref{proportiesFDSC}. Moreover, except $ u'_{d+1} $ and $ v'_{d+1} $, any neighbor of the centers of $ F_i $ and $F_j$  cannot have an external neighbor in $ G_1 $ by Lemma \ref{proportiesFDSC} (see Figure \ref{K1mrenkli} (6)) for one of the possible cases depending the roles $ u'_{d+1} $ and $ v'_{d+1} $ in $ F_i $ and $ F_j $, respectively). Thus, $ F_i\cap G_1=\emptyset $ and $ F_j\cap G_1=\emptyset $ in each of the four possibilities, this contradicts that $ |\mathcal{F}^1|= |\mathcal{F}|=\frac{d}{2}$. 
		\end{itemize}	
		Hence, in each of the three cases above, there exists a fault-free edge joining a vertex from $ C $ to a vertex from  $ FDSC_n-G_1-\mathcal{F} $.

		\item [$ (2) $]Let $ |V(C)|=1 $, such that $V(C)=\{w\}$. 
		
		\begin{itemize}
			\item [$ (a) $]  Let  $ w \notin \{u,v\}$. Suppose that $ w_{d+1} $ is in some $ F_i\in \mathcal{F}  $. By Lemma \ref{proportiesFDSC} (1), $ w_{d+1} \notin G_1 $. Let  $ w_{d+1} 
			\in G_j $, where $ j\in T-\{1\} $. Note that there is only one cross edge between $ G_1 $ and $ G_j $. Thus, no vertex of $ F_i $ is in $ G_1 $. That is, $ F_i\cap G_j\neq\emptyset $ and $ F_i\cap G_1=\emptyset $,  this contradicts that $ |\mathcal{F}^1|= |\mathcal{F}|=\frac{d}{2}$. Hence, $ w_{d+1} $ is not in $ V(\mathcal{F}) $. Then the fault-free edge $ (w, w_{d+1}) $ connects $ C $ and $ FDSC_n-G_1-\mathcal{F} $.	
			
			\item [$ (b) $]  Let  $ w \in \{u,v\}$. Assume that $ w=u $ ( resp., $ w=v $). 
			
			If $ u $ and $ v $ are adjacent in $G_1$, then $ v=\overline{B_1}B_1 \in V(\mathcal{F}^1) $ (resp., $ u=B_1B_1 \in V(\mathcal{F}^1)$) since $ |V(C)|=1 $.  If $ w_{d+1} $ is a faulty vertex, considering the neighbors of $u$ (resp., $v$) in $FDSC_\frac{n}{2}$ by the induction hypothesis, we have
			\[|\mathcal{F}|\geq (\lfloor\frac{d-1}{2}\rfloor+1)+1=\frac{d}{2}+1,\]
			where $d$ is even, which contradicts $ |\mathcal{F}|=\frac{d}{2}$. Thus, $ w_{d+1} $ is fault-free vertex, and so the fault-free edge $ (w, w_{d+1}) $ connects $ C $ and $ FDSC_n-G_1-\mathcal{F} $.

			If $ u $ and $ v $ are not adjacent in $G_1$, then $u_{d+1}$ and $v_{d+1}$ are not adjacent. Suppose that $ w_{d+1} $ is in some $ F_i\in \mathcal{F}  $. By Lemma \ref{proportiesFDSC} (1), $ w_{d+1} \notin G_1 $. Let  $ w_{d+1} 
			\in G_j $, where $ j\in T-\{1\} $. By Lemma \ref{proportiesFDSC}, no vertex of $ F_i $ is in $ G_1 $. That is, $ F_i\cap G_j\neq\emptyset $ and $ F_i\cap G_1=\emptyset $,  this contradicts that $ |\mathcal{F}^1|= |\mathcal{F}|=\frac{d}{2}$. Hence, $ w_{d+1} $ is fault-free vertex, and so the fault-free edge $ (w, w_{d+1}) $ connects $ C $ and $ FDSC_n-G_1-\mathcal{F} $.
		\end{itemize}
	\end{itemize}
	Thus, any component $C$ of $G_1$ is connected to $ FDSC_n-G_1-\mathcal{F} $. That is,  $ FDSC_n-\mathcal{F} $ is connected, a contradiction.
\end{proof}

\begin{lemma} \label{upperboundK1m}For $ d\geq1$ and $ n=2^d $,
	$\kappa(FDSC_n;K_{1,m})\leq \lfloor\frac{d}{2}\rfloor+1$ where $ 2\leq m \leq d+1 $.
\end{lemma}

\begin{proof}
	The proof is trivial for $d=1$. Let $d\geq 2$. It is enough to construct a $K_{1,m}$-structure-cut of cardinality $\lfloor\frac{d}{2}\rfloor+1$ to complete the proof.  Let $u=\overline{B_1}B_1$ and $u_2=B_1B_1$, where $B_1$ is the address of the module $G_1$. We consider the following two cases.
	\begin{figure}[h]
		\centering
		
		\includegraphics[width=1.005\textwidth]{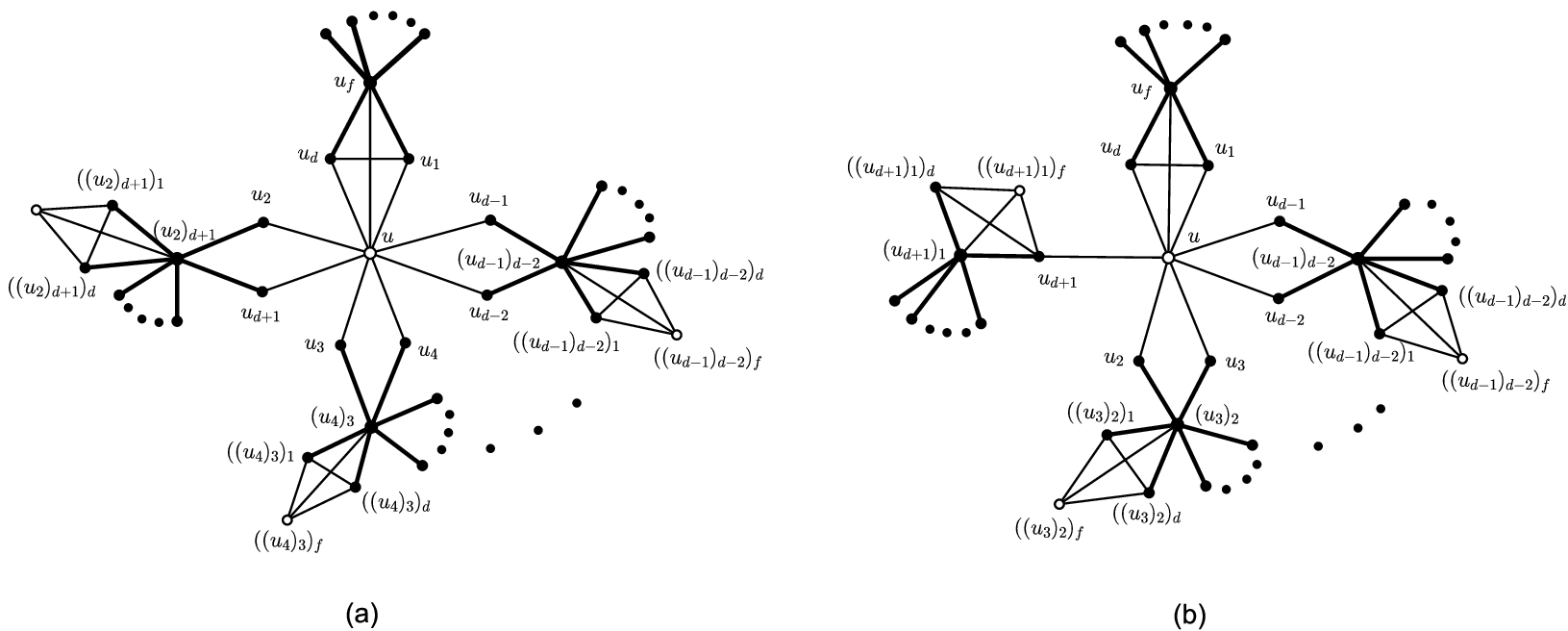}
		\caption{Explanation of Lemma \ref{upperboundK1m} when (a) $d$ is odd, (b) $d$ is even}
			\label{K1mupperbound}
	\end{figure}
	
	Let $d$ be odd. Let $F_1$ be a $K_{1,m}$ with the center $(u_2)_{d+1}$ and the leaf set which is obtained by the union of $\{u_2, u_{d+1}\}$ and the set of any $m-2$ vertices from $N_{FDSC_n}((u_2)_{d+1})-\{u_2, u_{d+1}, ((u_2)_{d+1})_f\}$. Likewise, let $F_{\frac{j}{2}}$ be a $K_{1,m}$ with the center $(u_j)_{j-1}$ and the leaf set which is obtained by the union of $\{ u_j, u_{j-1}\}$ and the set of any $m-2$ vertices from $N_{FDSC_n}((u_j)_{j-1})-\{u_j, u_{j-1}, ((u_j)_{j-1})_f\}$, for each even $j \in \{4, \dots, d-1\}$. Let $F_{\frac{d-1}{2}+1}$ be the $K_{1,m}$ with the center $u_f$ and the leaf set $\{u_1,u_d\}\cup \{(u_f)_j~\vert ~j\in[m+1]-\{1,d,f\}\}$ (see Figure \ref{K1mupperbound} (a)). If we consider  $\mathcal{F}=\{F_1,\ldots,F_{\frac{d-1}{2}+1}\}$, then
	\[|\mathcal{F}|=\frac{d-1}{2}+1=\lfloor\frac{d}{2}\rfloor+1.\]
	
	Let $d$ be even.  Let $F_{\lfloor\frac{j}{2}\rfloor}$ be a $K_{1,m}$ with the center $(u_j)_{j-1}$ and the leaf set which is obtained by the union of  $\{ u_j, u_{j-1}\}$ and the set of any $m-2$ vertices from $N_{FDSC_n}((u_j)_{j-1})-\{u_j, u_{j-1},((u_j)_{j-1})_f\}$ for each odd $j \in \{3, \dots, d-1\}$. Let $F_{\lfloor\frac{d-1}{2}\rfloor+1}$ be a $K_{1,m}$ with the center $(u_{d+1})_1$ and the leaf set which is obtained by the union of $\{u_{d+1}, ((u_{d+1})_1)_d\}$ and the set of any $m-2$ vertices from $N_{FDSC_n}((u_{d+1})_1)-\{u_{d+1}, (u_{d+1})_1)_d,((u_{d+1})_1)_f\}$. Let $F_{\lfloor\frac{d-1}{2}\rfloor+2}$  be the $K_{1,m}$ with the center $u_f$ and the leaf set $\{u_1,u_d\}\cup \{(u_f)_j~\vert ~j\in[m+1]-\{1,d,f\}\}$ (see Figure \ref{K1mupperbound} (b)).
	If we consider $\mathcal{F}=\{F_1,\ldots,F_{\lfloor\frac{d-1}{2}\rfloor+2}\}$, then
	\[|\mathcal{F}|=\lfloor\frac{d-1}{2}\rfloor+2=\frac{d}{2}+1.\]
	
	In both cases $|\mathcal{F}|=\lfloor\frac{d}{2}\rfloor+1$ and $u$ is an isolated vertex in $FDSC_n-\mathcal{F}$. Thus, $\mathcal{F}$ is a $K_{1,m}$-structure-cut  of $FDSC_n$. That is, $\kappa(FDSC_n;K_{1,m})\leq \lfloor\frac{d}{2}\rfloor+1$ for $d\geq 1$ and $n=2^d$ where $ 2\leq m \leq d+1 $.
\end{proof}

By Lemma \ref{lowerboundK1m}, Lemma \ref{upperboundK1m} and the fact that $\kappa(FDSC_n;K_{1,m})\geq \kappa^s(FDSC_n;K_{1,m})$, we have the following result.

\begin{theorem} \label{K1m} For $ d\geq1 $ and $ n=2^d $,	$\kappa(FDSC_n;K_{1,m})=\kappa^s(FDSC_n;K_{1,m})=\lfloor\frac{d}{2}\rfloor+1$ where $ 2\leq m \leq d+1$.
\end{theorem}

Since $K_{1,m}$ is a connected subgraph of $K_{1,d+2}$ for $m \in \{2,\dots, d+1\}$,  we have $\kappa^s(FDSC_n;K_{1,d+2})\leq\kappa^s(FDSC_n;K_{1,m})$. 
We can fully determine the $K_{1,m}$-substructure connectivity of $FDSC_n$ by combining this inequality with Lemma \ref{lowerboundK1m} and Theorem \ref{K1m}.

\begin{theorem}  \label{kappaSK1m} For $ d\geq1 $ and $ n=2^d $,	$\kappa^s(FDSC_n;K_{1,m})=\lfloor\frac{d}{2}\rfloor+1$ where $ 2\leq m \leq d+2$.
\end{theorem}

\section{Conclusion}

In this paper, we investigate the $ \mathcal{H} $-structure connectivity and $ \mathcal{H} $-substructure connectivity of a recently introduced hypercube variant, namely folded divide-and-swap cube $ FDSC_n $ for $ \mathcal{H}\in\{K_1, K_{1,1}, K_{1,m} (2\leq m \leq d+1) \} $ where  $ d\geq1 $ and $ n=2^d $. Due to its desirable properties, the folded divide-and-swap cube is a suitable interconnection network for large-scale multi-computer systems. In light of this, it would be of particular interest to conduct further reliability studies on this class. 

\vspace{2cc}
\noindent \textbf{Acknowledgment.} This work was supported by TÜBİTAK (Scientific and Technological Research Council of Türkiye) under the 1002 Project (Grant No. 122F276). 

\vspace{2cc}
\bibliographystyle{abbrv}
\bibliography{references}

\end{document}